\documentclass{article}

\usepackage{amsmath,amssymb,amsfonts,amsthm}

\newtheorem{theorem}{Theorem}[section]
\newtheorem{proposition}[theorem]{Proposition}
\newtheorem{lemma}[theorem]{Lemma}

\numberwithin{equation}{section}

\theoremstyle{remark}
\newtheorem{remark}[theorem]{Remark}

\newcommand{\II}{\mathop{\mathrm{II}}\nolimits}

\newcommand{\Ric}{\mathop{\mathrm{Ric}}\nolimits}

\title{Metrics with prescribed Ricci curvature near the boundary of a manifold}

\author{Artem Pulemotov\thanks{School of Mathematics and Physics, The
University of Queensland, St Lucia,~QLD 4072,
Australia}~\thanks{Department of Mathematics, The University of
Chicago, 5734 South University Ave, Chicago,~IL
60637-1514, USA} \\
\small{\texttt{a.pulemotov@uq.edu.au}}}

\begin{document}

\maketitle

\begin{abstract}
Suppose $M$ is a manifold with boundary. Choose a point
$o\in\partial M$. We investigate the prescribed Ricci curvature
equation $\Ric(G)=T$ in a neighborhood of $o$ under natural boundary
conditions. The unknown $G$ here is a Riemannian metric. The letter
$T$ on the right-hand side denotes a (0,2)-tensor. Our main theorems
address the questions of the existence and the uniqueness of
solutions. We explain, among other things, how these theorems may be
used to study rotationally symmetric metrics near the boundary of a
solid torus $\mathcal T$. The paper concludes with a brief
discussion of the Einstein equation on $\mathcal T$.
\end{abstract}

\section{Introduction}

The main theme of the present paper is the prescribed Ricci
curvature equation. We begin with a brief historical review. Suppose
$M$ is a closed manifold. Let~$T$ be a $(0,2)$-tensor on $M$.
Consider the prescribed Ricci curvature equation
\begin{align}\label{intro_RicciEq}
\Ric(G)=T
\end{align}
for a Riemannian metric $G$ on $M$. Whether or
not~\eqref{intro_RicciEq} has a solution is a significant question
in geometric analysis. The first major step towards answering this
question was taken by D.~DeTurck. More specifically, assume the
tensor $T$ is nondegenerate at a point $o\in M$. Then it is possible
to solve~\eqref{intro_RicciEq} in a neighborhood of~$o$. This result
was originally established in D.~DeTurck's paper~\cite{DDT81}.
Alternative proofs appeared later in several sources
including~\cite{AB87,JKweb}. One might not be able, however, to
solve equation~\eqref{intro_RicciEq} on \emph{all} of~$M$. A strong
nonexistence theorem for~\eqref{intro_RicciEq} was offered by
D.~DeTurck and N.~Koiso in~\cite{DDTNK84}. It tells us that,
whenever $T$ is positive-definite, there is a constant $c_T>0$ such
that $c_TT$ is not the Ricci curvature of any Riemannian metric
on~$M$. The books~\cite{JK85,AB87,TA98} contain rather detailed
surveys of the results discussed in this paragraph. For related
work, check out~\cite{DDT83,AB86,ED02,PhD03,RPKT09} and references
therein.

Whether a solution of~\eqref{intro_RicciEq} is unique in any sense
is also an important question in geometric analysis. Progress on
this question was made in many papers such
as~\cite{RH84,DDTNK84,XX92,JCDDT94,ED02,RPKT06}. As of today,
however, a complete answer is still lacking. The reader may
consult~\cite{JK85,AB87,TA98} for surveys of some of the results.
This concludes our historical review. A little more information will
be given in the end of Section~\ref{sec_prescribed}. Meanwhile, we
are ready to explain the problem we intend to investigate in the
present paper.

Let us change our setup a little. From now on, we suppose $M$ is a
manifold with boundary~$\partial M$. As before, we choose a
$(0,2)$-tensor $T$ on $M$. Consider the prescribed Ricci curvature
equation~\eqref{intro_RicciEq} on~$M$. A new question arises: does
this equation have a solution satisfying interesting boundary
conditions? Assume $T$ is nondegenerate. The main result
of~\cite{DDT81} implies that, given a point in the interior of $M$,
one can find a neighborhood $V$ of this point and a Riemannian
metric $G$ on $V$ such that~\eqref{intro_RicciEq} holds in $V$. On
the other hand, the nonexistence theorem of~\cite{DDTNK84} suggests
that, at least for positive-definite $T$, it may be problematic to
solve~\eqref{intro_RicciEq} on all of $M$; cf.
Remark~\ref{rem_global} below. Thus, it seems natural to refine the
question we just asked as follows: given $o\in\partial M$, can one
find a neighborhood $V$ of $o$ and a Riemannian metric $G$ such
that~\eqref{intro_RicciEq} holds in $V$ and $G$ satisfies
interesting boundary conditions on $V\cap\partial M$? Also, it is
important to understand what can be said about the uniqueness of
$G$. Up until now, not much was known about~\eqref{intro_RicciEq} on
manifolds with boundary (however, see~\cite{MAMH08}). Related topics
were studied in~\cite{DDT83,MA08,MAMH08,AP10} and several other
works. In particular, one indication of appropriate boundary
conditions comes from general relativity. This is demonstrated
by~\cite{DDT83} and references therein (specifically, see the papers
by Y.~Choquet-Bruhat).

Choose a neighborhood $U$ of $o\in\partial M$. Suppose
$(r,q_1,\ldots,q_{n-1})$ is a coordinate system in $U$ with $r$
taking values in $[0,\infty)$ and $q_1,\ldots,q_{n-1}$ taking values
in $\mathbb R$. Thus, $\frac\partial{\partial r}$ is transverse to
the boundary on $U\cap\partial M$. This paper addresses the
questions raised in the previous paragraph stipulating that the
components of the Riemannian metric $G$ with respect to
$(r,q_1,\ldots,q_{n-1})$ take a certain simple form. Related
research was carried out by J.~Cao and D.~DeTurck in~\cite{JCDDT94}.
The reader will find a brief discussion of this in the end of
Section~\ref{sec_prescribed}. Meanwhile, let us explain our main
results in more detail.

We consider the prescribed Ricci curvature
equation~\eqref{intro_RicciEq} near $o\in\partial M$. The boundary
conditions we impose are
\begin{align}\label{intro_BC}
G_{\partial M}=R,~\II(G)=S;
\end{align}
cf.~\cite{DDT83}. Here, $G_{\partial M}$ is the metric induced by
$G$ on $U\cap\partial M$, and $\II(G)$ is the second fundamental
form of $U\cap\partial M$ computed in~$G$ with respect to the
outward unit normal. The right-hand sides, $R$ and $S$, are
$(0,2)$-tensors on $U\cap\partial M$. It is clear that $R$ must be
positive-definite. We restrict our attention to Riemannian metrics
near $o\in\partial M$ that satisfy two requirements: Firstly, they
are diagonal in the coordinates $(r,q_1,\ldots,q_{n-1})$. Secondly,
their components in these coordinates do not depend on
$q_1,\ldots,q_{n-1}$. Such Riemannian metrics will be called type
$\mathcal A$ metrics near $o\in\partial M$. Further in the text, we
will discuss the prescribed Ricci curvature equation on a solid
torus. That discussion will furnish geometrically meaningful
examples of type~$\mathcal A$ metrics near $o\in\partial M$.

Now we can summarize the main results of the present paper. They are
stated in Section~\ref{sec_prescribed} as
Theorems~\ref{prop_nec_cond}, \ref{thm_loc_ex}, and \ref{thm_uniq}.
Suppose the tensor $T$ on the right-hand side
of~\eqref{intro_RicciEq} is nondegenerate (actually, it suffices for
our purposes to impose an assumption that is weaker but more
difficult to formulate). Theorems~\ref{prop_nec_cond}
and~\ref{thm_loc_ex} provide a necessary and sufficient condition
for the existence of a type $\mathcal A$ metric on a neighborhood of
$o\in\partial M$ solving~\eqref{intro_RicciEq} and
satisfying~\eqref{intro_BC}. This condition is exceedingly easy to
verify. Theorem~\ref{thm_uniq} shows that any two type $\mathcal A$
metrics near~$o\in\partial M$ solving~\eqref{intro_RicciEq} and
satisfying~\eqref{intro_BC} must be the same; cf.~\cite{MAMH08}. One
more comment is in order here. The methods in
Section~\ref{sec_prescribed} can be used to investigate invariant
metrics with prescribed Ricci curvature on cohomogeneity one
manifolds. We explain this in Remark~\ref{rem_cohom1} below.

Consider a solid torus $\mathcal T$. Section~\ref{sec_prescr_torus}
demonstrates how our Theorems~\ref{prop_nec_cond}, \ref{thm_loc_ex},
and \ref{thm_uniq} help study rotationally symmetric metrics on a
neighborhood of $\partial\mathcal T$. One can check that such
metrics are type $\mathcal A$ metrics near every point of
$\partial\mathcal T$ in properly chosen coordinates. They often
occur in applications; see, e.g.,~\cite{JDDKAY10} and references
therein.

The paper ends with a brief look at the Einstein equation. The
solutions of this equation in the class of rotationally symmetric
metrics on $\mathcal T$ can be found explicitly through a simple
computation (note that they are metrics with constant sectional
curvature since $\mathcal T$ is 3-dimensional). We list them in
Section~\ref{sec_Einstein}, Proposition~\ref{prop_Einstein_+}. There
is a connection between the content of that section and the topics
discussed in~\cite{MA08,ADMW99}.

N.B.: Throughout the paper, we deal with smooth tensors. This seems
to be the most natural way to present the material. It is possible,
however, to modify our results so that they apply to tensors with
weaker differentiability properties. We leave the details to the
reader.

\section{Tensors near the boundary}\label{sec_prescribed}

Suppose $M$ is a smooth oriented $n$-dimensional ($n\ge2$) manifold
with boundary $\partial M$. Our principal goal is to study the
prescribed Ricci curvature equation on $M$. In this section, we
present the notation, review the necessary background, and state two
lemmas.

Choose a point $o\in\partial M$. Assume $U$ is a neighborhood of $o$
and $(r,q_1,\ldots,q_{n-1})$ is a coordinate system in $U$ centered
at $o$. The parameter $r$ takes values in $[0,\infty)$, while
$q_1,\ldots,q_{n-1}$ takes values in $\mathbb R$. The intersection
$U\cap\partial M$ consists of the points $(0,q_1,\ldots,q_{n-1})$ in
$U$ for which $q_1,\ldots,q_{n-1}\in\mathbb R$. Given
$x\in(0,\infty)$, define the neighborhood $U_x$ of $o$ in $M$ to be
the set of $(r,q_1,\ldots,q_{n-1})$ with $r\in[0,x)$ and
$q_1,\ldots,q_{n-1}\in\mathbb R$. We say that a smooth (0,2)-tensor
$T$ on $U_x$ is a \emph{type $\mathcal A$ tensor on $U_x$} if
\begin{align}\label{rot_sym_tensor}
T=\sigma(r)\,dr\otimes dr+\sum_{i=1}^{n-1}\phi_i(r)\,dq_i\otimes
dq_i,\qquad r\in[0,x),
\end{align}
for functions $\sigma$ and $\phi_1,\ldots,\phi_{n-1}$ from $[0,x)$
to $\mathbb R$. Thus, $T$ has to be diagonal in the coordinates
$(r,q_1,\ldots,q_{n-1})$, and its components have to be independent
of $q_1,\ldots,q_{n-1}$. Accordingly, a smooth Riemannian metric $G$
on $U_x$ is a \emph{type $\mathcal A$ metric on $U_x$} if
\begin{align}\label{rot_sym_metric}
G=h^2(r)\,dr\otimes dr+\sum_{i=1}^{n-1}f_i^2(r)\,dq_i\otimes
dq_i,\qquad r\in[0,x),
\end{align}
where $h$ and $f_1,\ldots,f_{n-1}$ are functions from $[0,x)$ to
$(0,\infty)$. We will now state two computational lemmas, which we
will use throughout the paper.

Suppose $G$ is a type $\mathcal A$ metric on $U_x$ given
by~\eqref{rot_sym_metric}. Its Ricci curvature will be denoted by
$\Ric(G)$. The following lemma expresses $\Ric(G)$ in terms of $h$
and $f_1,\ldots,f_{n-1}$. Originally, a version of this lemma for
$n=3$ was provided to us by Andrea Young. Related computations may
be found in, e.g.,~\cite{RH84,CB98,ADMW99,JHEMW00,JDDKAY10}
and~\cite[Chapter~9]{AB87}. In this section and in
Section~\ref{sec_prescribed}, the prime designates differentiation
with respect to $r$.

\begin{lemma}\label{lem_Ricci_curv}
The Ricci curvature of the metric $G$ given
by~\eqref{rot_sym_metric} satisfies the equality
\begin{align*}
\Ric(G)=\sum_{j=1}^{n-1}\left(-\frac{f_j''}{f_j}+\frac{h'f_j'}{hf_j}\right)dr\otimes
dr+\sum_{i=1}^{n-1}\Biggl(-\frac{f_if_i''}{h^2}+\frac{h'f_if_i'}{h^3}+\frac{(f_i')^2}{h^2}
-\sum_{j=1}^{n-1}\frac{f_if_j'f_i'}{h^2f_j}\Biggr)dq_i\otimes dq_i
\end{align*}
on the set $U_x$.
\end{lemma}

\begin{proof}
Direct calculation.
\end{proof}

The metric $G$ induces a Riemannian metric $G_{\partial M}$ on
$U\cap\partial M$. It is evident that
\begin{align}\label{ind_metric}
G_{\partial M}=\sum_{i=1}^{n-1}f_i^2(0)\,dq_i\otimes dq_i.
\end{align}
We denote by $\II(G)$ the second fundamental form of $U\cap\partial
M$ computed in the metric $G$ with respect to the outward unit
normal. It is easy to express $\II(G)$ in terms of $h$ and
$f_1,\ldots,f_{n-1}$.

\begin{lemma}\label{lem_2ff}
The equality
\begin{align*}
\II(G)=-\sum_{i=1}^{n-1}\frac{f_i(0)f_i'(0)}{h(0)}\,dq_i\otimes dq_i
\end{align*}
holds.
\end{lemma}

\begin{proof}
Another direct calculation.
\end{proof}

\section{Prescribed Ricci curvature}\label{sec_prescribed}

In this section, we discuss the existence of solutions to the
prescribed Ricci curvature equation in the class of type $\mathcal
A$ metrics on a neighborhood of the point $o\in\partial M$. Also, we
will address the issue of uniqueness. More precisely, fix
$x\in(0,\infty)$ and consider a tensor $T$ on $U_x$ given by
formula~\eqref{rot_sym_tensor}. Suppose
$\alpha_1,\ldots,\alpha_{n-1}$ and $\eta_1,\ldots,\eta_{n-1}$ are
real numbers with $\alpha_1,\ldots,\alpha_{n-1}>0$. We introduce the
(0,2)-tensors $R$ and $S$ on $U\cap\partial M$ by setting
\begin{align}\label{def_RS}
R=\sum_{i=1}^{n-1}\alpha_i^2\,dq_i\otimes dq_i,\qquad
S=\sum_{i=1}^{n-1}\eta_i\,dq_i\otimes dq_i.
\end{align}
Our goal is to answer the following questions:
\begin{enumerate}
\item
Is it possible to find, for some $\epsilon_0\in(0,x)$, a type
$\mathcal A$ metric $G$ on $U_{\epsilon_0}$ such that the Ricci
curvature of $G$ equals~$T$, the induced metric $G_{\partial M}$
equals~$R$, and the second fundamental form $\II(G)$ coincides
with~$S$?
\item
When such a $G$ exists, is it unique?
\end{enumerate}
In the spirit of~\cite{DDT81}, we will impose a nondegeneracy-type
assumption on~$T$. More specifically, we will require that $\sigma$
be bounded away from~0. Our first result is a relatively simple
necessary condition for the existence of~$G$. We will demonstrate
below that this condition is also sufficient.

\begin{theorem}\label{prop_nec_cond}
Let $G$ be a type $\mathcal A$ metric on $U_x$ such that $\Ric(G)=T$
with the tensor $T$ given by~\eqref{rot_sym_tensor}.
Assume~$\sigma(0)$ is not equal to~$0$. Suppose that $G_{\partial
M}=R$ and $\II(G)=S$ on $U\cap\partial M$ with the tensors $R$ and
$S$ defined by~\eqref{def_RS}. Then the quantity
\begin{align*}
\frac1{\sigma(0)}\sum_{j=1}^{n-1}\Bigg(\frac{1}{\alpha_j^2}\phi_j(0)
+2\sum_{k=j+1}^{n-1}\frac{\eta_k\eta_j}{\alpha_k^2\alpha_j^2}\Bigg)
\end{align*}
is positive.
\end{theorem}

\begin{proof}
Let us write $G$ in the form~\eqref{rot_sym_metric}. According to
Lemma~\ref{lem_Ricci_curv}, the fact that $\Ric(G)=T$ translates as
\begin{align}\label{main_ODE_sys}
\sum_{j=1}^{n-1}\left(-\frac{f_j''}{f_j}+\frac{h'f_j'}{hf_j}\right)&=\sigma,\notag\\
-\frac{f_if_i''}{h^2}+\frac{h'f_if_i'}{h^3}+\frac{(f_i')^2}{h^2}
-\sum_{j=1}^{n-1}\frac{f_if_j'f_i'}{h^2f_j}&=\phi_i,\qquad
i=1,\ldots,n-1.
\end{align}
We find $f_1'',\ldots,f_{n-1}''$ from the equations involving
$\phi_1,\ldots,\phi_{n-1}$ and substitute the results into the
equation involving $\sigma$. This yields
\begin{align*}
\sum_{j=1}^{n-1}\Bigg(\frac{h^2}{f_j^2}\phi_j+2\sum_{k=j+1}^{n-1}\frac{f_k'f_j'}{f_kf_j}\Bigg)=\sigma
\end{align*}
on $[0,x)$. Using formula~\eqref{ind_metric} and
Lemma~\ref{lem_2ff}, we derive
\begin{align*}
\sum_{j=1}^{n-1}\Bigg(\frac{h^2(0)}{\alpha_j^2}\phi_j(0)+2\sum_{k=j+1}^{n-1}\frac{h^2(0)\eta_k\eta_j}{\alpha_k^2\alpha_j^2}\Bigg)=\sigma(0).
\end{align*}
Consequently,
\begin{align}\label{fnl_aux_fla1}
\frac1{\sigma(0)}\sum_{j=1}^{n-1}\Bigg(\frac{1}{\alpha_j^2}\phi_j(0)
+2\sum_{k=j+1}^{n-1}\frac{\eta_k\eta_j}{\alpha_k^2\alpha_j^2}\Bigg)&=\frac1{h^2(0)}>0.
\end{align}
\end{proof}

\begin{remark}
Suppose the conditions of Theorem~\ref{prop_nec_cond} are satisfied.
Then the outward unit normal vector field on $U\cap\partial M$ with
respect to the metric $G$ is equal to
\begin{align*}
-\frac1{h(0)}\,\frac\partial{\partial
r}=-\sqrt{\frac1{\sigma(0)}\sum_{j=1}^{n-1}\Bigg(\frac1{\alpha_j^2}\phi_j(0)
+2\sum_{k=j+1}^{n-1}\frac{\eta_k\eta_j}{\alpha_k^2\alpha_j^2}\Bigg)}\,\frac\partial{\partial
r}\,.
\end{align*}
This is a simple consequence of formula~\eqref{fnl_aux_fla1} in the
proof above.
\end{remark}

\begin{remark}\label{rem_Gauss-Codazzi}
It is possible to derive Theorem~\ref{prop_nec_cond} from the Gauss
equation. Indeed, together with our assumptions and the fact that
$R$ is a flat metric on $U\cap\partial M$, this equation implies
\begin{align*}
\sum_{i=1}^{n-1}\frac1{\alpha_i^2}\phi_i(0)
=\frac{\sigma(0)}{h^2(0)}+\sum_{i=1}^{n-1}\frac{\eta_i^2}{\alpha_i^4}-\Bigg(\sum_{i=1}^{n-1}\frac{\eta_i}{\alpha_i^2}\Bigg)^2.
\end{align*}
An elementary computation then leads to
formula~\eqref{fnl_aux_fla1}.
\end{remark}

\begin{remark}
As in Theorem~\ref{prop_nec_cond}, suppose $G$ is a type $\mathcal
A$ metric on $U_x$ such that $\Ric(G)=T$, $G_{\partial M}=R$, and
$\II(G)=S$ with $T$, $R$, and $S$ given by~\eqref{rot_sym_tensor}
and~\eqref{def_RS}. Assume that $\sigma(0)=0$. Then
\begin{align*}
\sum_{j=1}^{n-1}\Bigg(\frac{1}{\alpha_j^2}\phi_j(0)
+2\sum_{k=j+1}^{n-1}\frac{\eta_k\eta_j}{\alpha_k^2\alpha_j^2}\Bigg)=0.
\end{align*}
One can verify this easily by retracing the above proof of
Theorem~\ref{prop_nec_cond}.
\end{remark}

We are ready to formulate our next result. It will complete our
discussion of the existence of $G$.

\begin{theorem}\label{thm_loc_ex}
Let $T$ be a type $\mathcal A$ tensor on $U_x$ given by
formula~\eqref{rot_sym_tensor}. Suppose $\sigma(0)$ is not equal
to~$0$. Let $R$ and $S$ be tensors on $U\cap\partial M$ defined
by~\eqref{def_RS}. Assume that
\begin{align}\label{assu}
\frac1{\sigma(0)}\sum_{j=1}^{n-1}\Bigg(\frac{1}{\alpha_j^2}\phi_j(0)
+2\sum_{k=j+1}^{n-1}\frac{\eta_k\eta_j}{\alpha_k^2\alpha_j^2}\Bigg)>0.
\end{align}
Then, for some $\epsilon_0\in(0,x)$, there exists a type $\mathcal
A$ metric $G$ on $U_{\epsilon_0}$ satisfying the equality
$\Ric(G)=T$ on $U_{\epsilon_0}$ together with the equalities
$G_{\partial M}=R$ and $\II(G)=S$ on $U\cap\partial M$.
\end{theorem}

\begin{proof}
The argument we use may be viewed as a variant of D.~DeTurck's
argument from~\cite[Chapter~5]{AB87}. We elaborate on this in
Remark~\ref{rem_DeTurck_exi} below. Meanwhile, our goal is to find a
type $\mathcal A$ metric $G$ near $o\in\partial M$ such that
$\Ric(G)=T$, $G_{\partial M}=R$, and $\II(G)=S$. In order to do so,
we first consider the system
\begin{align}\label{equ_nondeg2}
\hat h''&=\sqrt{\frac1{\sigma\circ\hat h'}\sum_{j=1}^{n-1}\Bigg(\frac{\hat h^2}{\hat f_j^2}(\phi_j\circ\hat h')
+2\sum_{k=j+1}^{n-1}\frac{\hat f_k'\hat f_j'}{\hat f_k\hat f_j}\Bigg)}\,,\notag\\
\hat f_i''&=\frac{\hat h^2}{\hat f_i}\Biggl(-\phi_i\circ \hat
h'+\frac{\hat h'\hat f_i\hat f_i'}{\hat h^3}+\frac{(\hat
f_i')^2}{\hat h^2} -\sum_{j=1}^{n-1}\frac{\hat f_i\hat f_j'\hat
f_i'}{\hat h^2\hat f_j}\Biggr).
\end{align}
Here, $i$ takes the values $1,\ldots,n-1$, while the unknown $\hat
h$ and $\hat f_1,\ldots,\hat f_{n-1}$ are real-valued functions of
the parameter $r\in[0,x)$. We impose the initial conditions by
requiring that
\begin{align}\label{term_cond_nondeg1}
\hat h(0)=1,~\hat f_i(0)=\alpha_i,~\hat h'(0)=0,~\hat
f_i'(0)=-\frac{\eta_i}{\alpha_i}\,,\qquad i=1,\ldots,n-1.
\end{align}
These conditions, together with~\eqref{assu}, ensure that the
right-hand sides of equations~\eqref{equ_nondeg2} are well-defined
when $r=0$. Further, the expression under the square root symbol is
positive at $r=0$. Employing the standard Picard-Lindel\"of
existence theory for ordinary differential equations, one can
demonstrate that
problem~\eqref{equ_nondeg2}--\eqref{term_cond_nondeg1} has a
solution on the interval $[0,\epsilon]$ for some $\epsilon\in(0,x)$.
More specifically, there are smooth functions $\hat h$ and $\hat
f_1,\ldots,\hat f_{n-1}$ on $[0,\epsilon]$ such that
formulas~\eqref{equ_nondeg2} hold on $[0,\epsilon]$ and
formulas~\eqref{term_cond_nondeg1} hold as well. In particular,
these functions are not~0 on $[0,\epsilon]$. The values of $\hat h'$
on $[0,\epsilon]$ lie in $[0,x)$. Note that $\hat h$ and $\hat
f_1,\ldots,\hat f_{n-1}$ are all positive at $r=0$. The same can be
said about the second derivative $\hat h''$. Therefore, $\hat h$ and
$\hat f_1,\ldots,\hat f_{n-1}$ are positive on $[0,\epsilon]$, and
we may assume that $\epsilon$ is small enough to guarantee that
$\hat h''$ is positive on $[0,\epsilon]$.

Let us proceed to construct $G$ near $o\in\partial M$ such that
$\Ric(G)=T$, $G_{\partial M}=R$, and $\II(G)=S$. With the functions
$\hat h$ and $\hat f_1,\ldots,\hat f_{n-1}$ at hand, we introduce
the metric $\hat G$ on $U_\epsilon$ according to the formula
\begin{align*}
\hat G=\hat h^2(r)\,dr\otimes dr+\sum_{i=1}^{n-1}\hat
f_i^2(r)\,dq_i\otimes dq_i,\qquad r\in[0,\epsilon).
\end{align*}
Denote $\epsilon_0=\hat h'(\epsilon)$. Consider the map
$\Theta:U_\epsilon\to U_{\epsilon_0}$ given in our coordinates by
\begin{align*}
\Theta\big((r,q_1,\ldots,q_{n-1})\big)=\big(\hat
h'(r),q_1,\ldots,q_{n-1}\big),\qquad
r\in[0,\epsilon),~q_1,\ldots,q_{n-1}\in\mathbb R.
\end{align*}
Since $\hat h''$ is positive on $[0,\epsilon)$, this map is a
diffeomorphism. We set $G=(\Theta^{-1})^*\hat G$, where the asterisk
designates pullback. It is clear that $G$ is a type $\mathcal A$
metric on $U_{\epsilon_0}$. To complete the proof, we need to show
that $\Ric(G)=T$, $G_{\partial M}=R$, and $\II(G)=S$.

Formulas~\eqref{equ_nondeg2} imply
\begin{align*}
\sum_{j=1}^{n-1}\biggl(-\frac{\hat f_j''}{\hat f_j}+\frac{\hat
h'\hat
f_j'}{\hat h\hat f_j}\biggr)&=(\hat h'')^2(\sigma\circ\hat h'),\notag\\
-\frac{\hat f_i\hat f_i''}{\hat h^2}+\frac{\hat h'\hat f_i\hat
f_i'}{\hat h^3}+\frac{(\hat f_i')^2}{\hat h^2}
-\sum_{j=1}^{n-1}\frac{\hat f_i\hat f_j'\hat f_i'}{\hat h^2\hat
f_j}&=\phi_i\circ \hat h',\qquad i=1,\ldots,n-1,
\end{align*}
when the variable $r$ takes values in $[0,\epsilon)$. Consequently,
in view of Lemma~\ref{lem_Ricci_curv}, the equality $\Ric(\hat
G)=\Theta^*T$ holds on $U_\epsilon$. We use this equality to
conclude that
\begin{align*}
\Ric(G)&=\Ric\big((\Theta^{-1})^*\hat
G\big)=(\Theta^{-1})^*\Ric(\hat G)=(\Theta^{-1})^*\Theta^*T=T
\end{align*}
on $U_{\epsilon_0}$. It remains to study the behavior of $G$ near
the boundary.

Due to conditions~\eqref{term_cond_nondeg1}, the metric $\hat
G_{\partial M}$ induced by $\hat G$ on $U\cap\partial M$ coincides
with~$R$. Employing~\eqref{term_cond_nondeg1} and
Lemma~\ref{lem_2ff}, we can also establish that $\II(\hat G)$ equals
$S$. Besides, the restriction of $\Theta$ to $U\cap\partial M$ is
the identity map. These facts imply that
\begin{align*}
G_{\partial M}&=(\Theta^{-1})^*\hat G_{\partial M}=\hat G_{\partial
M}=R,
\\ \II(G)&=\II\big((\Theta^{-1})^*\hat G\big)=(\Theta^{-1})^*\II(\hat
G)=\II(\hat G)=S.
\end{align*}
Thus, we have verified all the required properties of $G$.
\end{proof}

\begin{remark}
In the proof of Theorem~\ref{thm_loc_ex}, the components of $G$ can
be expressed through the components of $\hat G$. More specifically,
if $G$ is given by~\eqref{rot_sym_metric}, one easily sees that
\begin{align}\label{expl_nohat_hat}
h=\frac{\hat h}{\hat h''}\circ(\hat h')^{-1},~f_i=\hat f_i\circ(\hat
h')^{-1},\qquad i=1,\ldots,n-1.
\end{align}
When carrying out the proof, we derived the equalities $\Ric(G)=T$,
$G_{\partial M}=R$, and $\II(G)=S$ from the formula
$G=(\Theta^{-1})^*\hat G$. Alternatively, one could verify them
using Lemmas~\ref{lem_Ricci_curv} and~\ref{lem_2ff} together
with~\eqref{expl_nohat_hat}.
\end{remark}

\begin{remark}\label{rem_DeTurck_exi}
Chapter~5 of the book~\cite{AB87} describes a method, due to
D.~DeTurck, for obtaining Riemannian metrics with prescribed Ricci
curvature on a neighborhood of an interior point of a manifold. Our
proof of Theorem~\ref{thm_loc_ex} may be interpreted as an
implementation of that method in the framework of ordinary
differential equations.
\end{remark}

The next result establishes the uniqueness of $G$. Roughly speaking,
we will demonstrate that any type $\mathcal A$ metric $\tilde G$
near $o\in \partial M$ with the same Ricci curvature and the same
boundary behavior as $G$ must coincide with $G$. Note that the
paper~\cite{MAMH08} discusses a closely related statement.

\begin{theorem}\label{thm_uniq}
Let $T$ be a tensor on $U_x$ given by~\eqref{rot_sym_tensor}.
Suppose $\sigma(r)\ne0$ whenever $r\in[0,x)$. Let $G$ and $\tilde G$
be type $\mathcal A$ metrics on $U_x$ satisfying the equality
\begin{align*}
\Ric(G)=\Ric(\tilde G)=T
\end{align*} on $U_x$ together
with the equalities $G_{\partial M}=\tilde G_{\partial M}$ and
$\II(G)=\II(\tilde G)$ on $U\cap\partial M$. Then $G$ coincides with
$\tilde G$.
\end{theorem}

\begin{proof}
The argument we use was inspired by an argument proposed by
R.~Hamilton to establish the uniqueness of solutions to the Ricci
flow on a closed manifold. We will say more about this in
Remark~\ref{rem_Hamilton_uniq} below. Meanwhile, let us consider the
set
\begin{align*}
\Omega=\{0\}\cup\{y\in(0,x)\,|\,G=\tilde G~\mbox{on}~U_y\}
\end{align*} and denote
$y_0=\sup\Omega$. The proof of the theorem will be complete if we
show that $y_0$ equals $x$. Assume this is not the case. Then $y_0$
must be less than $x$. We will now use this fact to obtain a
contradiction. Our plan is to demonstrate that $G=\tilde G$ on
$U_{y_0+\delta_0}$ for some $\delta_0\in(0,x-y_0)$. This would
contradict the definition of~$y_0$.

Suppose $G$ has the form~\eqref{rot_sym_metric}. Consider the
ordinary differential equation
\begin{align}
\label{ODE_harm1}\hat h''=\frac{\hat h}{h\circ\hat h'}
\end{align}
on the interval $[y_0,x)$ for the unknown function $\hat h$. We
impose the initial conditions
\begin{align}\label{TC_harm1}
\hat h(y_0)=1,~\hat h'(y_0)=y_0.
\end{align}
Employing the standard theory of ordinary differential equations,
one can show that problem~\eqref{ODE_harm1}--\eqref{TC_harm1} has a
unique solution $\hat h$ on $[y_0,y_0+\delta]$ for some
$\delta\in(0,x-y_0)$. The reasoning one should use is the same as
the reasoning one used in the proof of Theorem~\ref{thm_loc_ex} to
deal with~\eqref{equ_nondeg2}--\eqref{term_cond_nondeg1}. The
solution $\hat h$ is smooth on $[y_0,y_0+\delta]$. The values of
$\hat h'$ lie in $[0,x)$. Also, $\hat h$ and $\hat h''$ are
positive. Let us denote $\hat f_i=f_i\circ \hat h'$ for
$i=1,\ldots,n-1$. These functions will help us show that $G=\tilde
G$ on an appropriate neighborhood of~$o\in\partial M$.

Set $\delta_1=\hat h'(y_0+\delta)-y_0$. Clearly,
$\delta_1\in(0,x-y_0)$. Our next step is to introduce a map $\Sigma$
acting from $U_{y_0+\delta}\setminus U_{y_0}$ to
$U_{y_0+\delta_1}\setminus U_{y_0}$ if $y_0\ne0$ and from
$U_{\delta}$ to $U_{\delta_1}$ if $y_0=0$. Define $\Sigma$ by the
equality
\begin{align*}
\Sigma\big((r,q_1,\ldots,q_{n-1})\big)&=\big(\hat
h'(r),q_1,\ldots,q_{n-1}\big), \qquad
r\in[y_0,y_0+\delta),~q_1,\ldots,q_{n-1}\in\mathbb R,
\end{align*}
employing our coordinates. It is easy to see that $\Sigma$ is a
diffeomorphism. Next we consider the metric $\hat G=\Sigma^*G$ (the
asterisk stands for pullback). Evidently, we can write this metric
in the form
\begin{align*}
\hat G=\hat h^2(r)\,dr\otimes dr+\sum_{i=1}^{n-1}\hat
f_i^2(r)\,dq_i\otimes dq_i,\qquad r\in[y_0,y_0+\delta).
\end{align*}
The equation $\Ric(\hat G)=\Sigma^*T$ holds. Together with
Lemma~\ref{lem_Ricci_curv} and the fact that $\hat h''$ must be
greater than~0 on $[y_0,y_0+\delta]$, this equation implies
\begin{align}\label{big_ODE_hat}
\hat h''&=\sqrt{\frac1{\sigma\circ\hat
h'}\sum_{j=1}^{n-1}\Bigg(\frac{\hat h^2}{\hat f_j^2}(\phi_j\circ\hat
h')
+2\sum_{k=j+1}^{n-1}\frac{\hat f_k'\hat f_j'}{\hat f_k\hat f_j}\Bigg)}\,,\notag\\
\hat f_i''&=\frac{\hat h^2}{\hat f_i}\Biggl(-\phi_i\circ \hat
h'+\frac{\hat h'\hat f_i\hat f_i'}{\hat h^3}+\frac{(\hat
f_i')^2}{\hat h^2} -\sum_{j=1}^{n-1}\frac{\hat f_i\hat f_j'\hat
f_i'}{\hat h^2\hat f_j}\Biggr),\qquad i=1,\ldots,n-1,
\end{align}
when the variable $r$ takes values in $[y_0,y_0+\delta)$. Also, it
is not difficult to understand that
\begin{align}\label{big_TC_hat}
\hat h(y_0)=1,~\hat f_i(y_0)=f_i(y_0),~\hat h'(y_0)=y_0,~\hat
f_i'(y_0)=\frac{f_i'(y_0)}{h(y_0)}\,,\qquad i=1,\ldots,n-1.
\end{align}
Roughly speaking, the initial-value
problem~\eqref{big_ODE_hat}--\eqref{big_TC_hat} enjoys the
uniqueness of solutions. We will use this to show that $G=\tilde G$
on $U_{y_0+\delta_0}$ for some $\delta_0\in(0,x-y_0)$.

Suppose the metric $\tilde G$ has the form
\begin{align*}
\tilde G=\tilde h^2(r)\,dr\otimes dr+\sum_{i=1}^{n-1}\tilde
f_i^2(r)\,dq_i\otimes dq_i,\qquad r\in[0,x),
\end{align*}
where $\tilde h$ and $\tilde f_1,\ldots,\tilde f_{n-1}$ are positive
functions on $[0,x)$. By analogy with~\eqref{ODE_harm1}, let us
consider the equation
\begin{align}\label{ODE_har2}
\check h''&=\frac{\check h}{\tilde h\circ\check h'}
\end{align}
on $[y_0,x)$ for the unknown $\check h$. As in~\eqref{TC_harm1}, we
demand that
\begin{align}\label{TC_har2}
\check h(y_0)=1,~\check h'(y_0)=y_0.
\end{align}
Problem~\eqref{ODE_har2}--\eqref{TC_har2} has a unique solution
$\check h$ on $[y_0,y_0+\tilde\delta]$ for some
$\tilde\delta\in(0,x-y_0)$. Let us define $\check f_i=\tilde
f_i\circ\check h'$ with $i=1,\ldots,n-1$. Arguing as above, we find
equations for $\check h''$ and $\check f_1'',\ldots,\check
f_{n-1}''$ in terms of $\check h$ and $\check f_1,\ldots,\check
f_{n-1}$, the derivatives $\check h'$ and $\check f_1',\ldots,\check
f_{n-1}'$, and the functions $\sigma$ and
$\phi_1,\ldots,\phi_{n-1}$. The next step is to invoke the
definition of $y_0$ and the assumptions of the theorem. They imply
\begin{align*}
\check h(y_0)&=1,~\check f_i(y_0)=\tilde f_i(y_0)=f_i(y_0), \notag \\
\check h'(y_0)&=y_0,~\check f_i'(y_0)=\frac{\tilde f_i'(y_0)}{\tilde
h(y_0)}=\frac{f_i'(y_0)}{h(y_0)}\,,\qquad i=1,\ldots,n-1. \notag
\end{align*}
We conclude that formulas~\eqref{big_ODE_hat}--\eqref{big_TC_hat}
will hold on $[y_0,y_0+\tilde\delta)$ if we replace $\hat h$ and
$\hat f_1,\ldots,\hat f_{n-1}$ (as well as their derivatives) in
these formulas by $\check h$ and~$\check f_1,\ldots,\check f_{n-1}$
(and their corresponding derivatives). This enables us to prove,
with the aid of the standard uniqueness results for ordinary
differential equations, that $\hat h=\check h$ on
$[y_0,y_0+\min\{\delta,\tilde\delta\}]$ and $\hat f_i=\check f_i$ on
that interval for all $i=1,\ldots,n-1$. Consequently, we have
\begin{align*}
h&=\frac{\hat h}{\hat h''}\circ(\hat h')^{-1}=\frac{\check h}{\check
h''}\circ(\check h')^{-1}=\tilde h, \\
f_i&=\hat f_i\circ(\hat h')^{-1}=\check f_i\circ(\check
h')^{-1}=\tilde f_i, \qquad i=1,\ldots,n-1,
\end{align*}
when the variable $r$ takes values in $[y_0,y_0+\delta_0)$ with the
number $\delta_0\in(0,x-y_0)$ given by
\begin{align*}
\delta_0=\hat h'(y_0+\min\{\delta,\tilde\delta\})-y_0=\check
h'(y_0+\min\{\delta,\tilde\delta\})-y_0.
\end{align*}
It follows that $G=\tilde G$ on $U_{y_0+\delta_0}$. But this
contradicts the definition of $y_0$. Hence $y_0$ equals $x$, and $G$
coincides with $\tilde G$.
\end{proof}

\begin{remark}\label{rem_Hamilton_uniq}
Our proof of Theorem~\ref{thm_uniq} may be interpreted as an
adaptation of R.~Hamilton's proof from~\cite{RH95} (see also,
e.g.,~\cite[Chapter~3]{BCDK04}) of the uniqueness of solutions to
the Ricci flow on a closed manifold. Let us comment on one aspect of
this interpretation. R.~Hamilton's proof employed two harmonic map
heat flows. The counterparts of these flows in our argument are the
ordinary differential equations~\eqref{ODE_harm1}
and~\eqref{ODE_har2}.
\end{remark}

One more remark is in order. It concerns all the results in
Section~\ref{sec_prescribed}, not just Theorem~\ref{thm_uniq}.

\begin{remark}\label{rem_cohom1}
Let the manifold $M$ be connected and acted upon by a compact Lie
group $\mathcal G$ with cohomogeneity one; see, for
example,~\cite{CB98,ADMW00,JHEMW00}. Assume that the isotropy
representation of the principal orbit type of $M$ splits into
pairwise inequivalent irreducible summands. This assumption is quite
natural; cf., for instance,~\cite{ADMW00}. The questions of the
existence and the uniqueness of solutions to the prescribed Ricci
curvature equation in the class of $\mathcal G$-invariant Riemannian
metrics near a principal orbit reduce to the investigation of a
system resembling~\eqref{main_ODE_sys}. It seems that, following the
arguments above, one can obtain analogues of
Theorems~\ref{prop_nec_cond}, \ref{thm_loc_ex}, and~\ref{thm_uniq}
for such metrics.
\end{remark}

We mentioned briefly in the introduction that some of the material
in~\cite{JCDDT94} was related to the results in the present paper.
Let us elaborate on this. Consider the standard action of the
orthogonal group $SO(m)$ on the Euclidean space $\mathbb R^m$ for
some $m\ge3$. Along with other things, the work~\cite{JCDDT94}
studied the existence and the uniqueness of solutions to the
prescribed Ricci curvature equation among $SO(m)$-invariant
Riemannian metrics on $\mathbb R^m$; cf. Remark~\ref{rem_cohom1}
above. It is possible to show that such metrics are conformally
flat. The proofs of many of the statements in~\cite{JCDDT94}
exploited this fact.

\section{The case of a solid torus}\label{sec_prescr_torus}

We illustrate below how the results of Section~\ref{sec_prescribed}
can provide information about rotationally symmetric metrics near
the boundary of a solid torus. Let us begin with some notation and
some background. Consider the unit disk $D^2=\{(w,y)\in\mathbb
R^2\,|\,w^2+y^2\le1\}$ and the unit circle $S^1=\{(w,y)\in\mathbb
R^2\,|\,w^2+y^2=1\}$ in $\mathbb R^2$. We will work with the solid
torus
\begin{align*}
\mathcal T=D^2\times S^1\subset\mathbb R^2\times\mathbb
R^2.\end{align*} The boundary of $\mathcal T$ will be denoted by
$\partial\mathcal T$. The set $\mathcal C=\{(0,0)\}\times
S^1\subset\mathcal T$ is the \emph{core circle} of~$\mathcal T$. We
employ cylindrical coordinates $(\gamma,\lambda,\mu)$ on $\mathcal
T$; cf., for instance,~\cite{JDDKAY10}. The parameter $\gamma$ takes
values in the interval $[0,1]$, while the parameters $\lambda$ and
$\mu$ takes values in $(0,1]$. The solid torus $\mathcal T$ is a
subset of $\mathbb R^2\times \mathbb R^2$, and the point
$(\gamma,\lambda,\mu)\in\mathcal T$ coincides with the point
$\big((w_1,y_1),(w_2,y_2)\big)\in\mathbb R^2\times \mathbb R^2$ such
that
\begin{align*}
w_{1}&=\gamma\cos(2\pi\lambda),~y_{1}=\gamma\sin(2\pi\lambda), \\
w_{2}&=\cos(2\pi\mu),~y_{2}=\sin(2\pi\mu).
\end{align*}
Our next step is to discuss two types of rotations of $\mathcal T$.
This will help us describe the tensors we will deal with in the
sequel.

For every $\lambda_0\in\mathbb (0,1]$, consider the map $\mathcal
V_{\lambda_0}:\mathcal T\to\mathcal T$ defined in cylindrical
coordinates by the formula
\begin{align*}
\mathcal
V_{\lambda_0}\big((\gamma,\lambda,\mu)\big)=\begin{cases}(\gamma,\lambda+\lambda_{0},\mu),&\mbox{if}~\lambda+\lambda_{0}\le1,
\\ (\gamma,\lambda+\lambda_{0}-1,\mu),&\mbox{if}~\lambda+\lambda_{0}>1.\end{cases}
\end{align*}
Intuitively, one may view $\mathcal V_{\lambda_0}$ as the rotation
of $\mathcal T$ by the angle $2\pi\lambda_0$ around the core circle
$\mathcal C$. It is easy to see that $\mathcal C$ remains fixed
under $\mathcal V_{\lambda_0}$. For each $\mu_0\in(0,1]$, consider
the map $\mathcal W_{\mu_0}:\mathcal T\to\mathcal T$ given by the
formula
\begin{align*}
\mathcal
W_{\mu_0}\big((\gamma,\lambda,\mu)\big)=\begin{cases}(\gamma,\lambda,\mu+\mu_0),&\mbox{if}~\mu+\mu_{0}\le1,
\\ (\gamma,\lambda,\mu+\mu_{0}-1),&\mbox{if}~\mu+\mu_{0}>1.\end{cases}
\end{align*}
If one visualizes $\mathcal T$ as a ``doughnut" in $\mathbb R^3$,
one may think of $\mathcal W_{\mu_0}$ as the rotation of $\mathcal
T$ by the angle $2\pi\mu_0$ around the axis passing through the
center of $\mathcal C$ perpendicular to the plane containing
$\mathcal C$. The set of fixed points of $\mathcal W_{\mu_0}$ is
empty unless $\mu_0=1$. In this section, we focus on rotationally
symmetric tensors on neighborhoods of $\partial\mathcal T$. By
definition, such tensors possess two properties. Firstly, they are
diagonal in the cylindrical coordinates $(\gamma,\lambda,\mu)$.
Secondly, they do not change when pulled back by $\mathcal
V_{\lambda_0}$ and $\mathcal W_{\mu_0}$ for any $\lambda_0\in(0,1]$
and $\mu_0\in(0,1]$. Rotationally symmetric tensors on neighborhoods
of $\partial\mathcal T$ admit a simple characterization. Roughly
speaking, their components in the coordinates $(\gamma,\lambda,\mu)$
do not depend on the parameters $\lambda$ and $\mu$. In the next
paragraph, we will talk about rotationally symmetric tensors near
$\partial\mathcal T$ in a more rigorous fashion. But before doing
so, we have to introduce one more piece of notation.

Fix $x\in(0,1]$. Let $\mathcal T_x$ stand for the set of points
$(\gamma,\lambda,\mu)$ in $\mathcal T$ such that $\gamma\in(1-x,1]$.
Thus, $\mathcal T_x$ is a neighborhood of~$\partial\mathcal T$ in
$\mathcal T$. We call a smooth (0,2)-tensor $T$ on $\mathcal T_x$ a
\emph{rotationally symmetric tensor on $\mathcal T_x$} if it is
given in cylindrical coordinates by the formula
\begin{align}\label{torus_t}
T=\zeta(\gamma)\,d\gamma\otimes
d\gamma+\psi_1(\gamma)\,d\lambda\otimes
d\lambda+\psi_2(\gamma)\,d\mu\otimes d\mu,\qquad \gamma\in(1-x,1].
\end{align}
On the right-hand side, $\zeta$ and $\psi_1,\psi_2$ have to be
functions from $(1-x,1]$ to $\mathbb R$. Granted that $T$
satisfies~\eqref{torus_t}, the equalities $\mathcal
V_{\lambda_0}^*T=T$ and $\mathcal W_{\mu_0}^*T=T$ hold for all
$\lambda_0\in(0,1]$ and $\mu_0\in(0,1]$. The asterisks in them
designate pullback. One easily sees that, when $T$
obeys~\eqref{torus_t}, it is a type $\mathcal A$ tensor on
appropriately chosen neighborhoods in $\mathcal T$ with appropriate
coordinate systems. A smooth Riemannian metric $G$ on $\mathcal T_x$
is a \emph{rotationally symmetric metric on $\mathcal T_x$} if
\begin{align}\label{torus_G}
G=p^2(\gamma)\,d\gamma\otimes d\gamma+g_1^2(\gamma)\,d\lambda\otimes
d\lambda+g_2^2(\gamma)\,d\mu\otimes d\mu,\qquad \gamma\in(1-x,1],
\end{align}
for functions $p$ and $g_1,g_2$ from $(1-x,1]$ to $(0,\infty)$. If
$G$ satisfies~\eqref{torus_G}, then it is a type $\mathcal A$ metric
on certain neighborhoods in $\mathcal T$ equipped with proper
coordinates. We write $\Ric(G)$ for the Ricci curvature of $G$. The
notation $G_{\partial\mathcal T}$ stands for the metric induced by
$G$ on $\partial\mathcal T$. Finally, $\II(G)$ is the second
fundamental form of $\partial\mathcal T$ computed in $G$ with
respect to the outward unit normal. Our goal is to show how the
results of Section~\ref{sec_prescribed} can help establish the
existence and the uniqueness of solutions to the prescribed Ricci
curvature equation in the class of rotationally symmetric metrics
near $\partial\mathcal T$. Only a little more preparation is
required at this point.

Assume $\beta_1,\beta_2$ and $\theta_1,\theta_2$ are real numbers
with $\beta_1,\beta_2>0$. Consider the tensors $R$ and $S$ on
$\partial\mathcal T$ defined by the equalities
\begin{align}\label{torus_RS}
R=\beta_1^2\,d\lambda\otimes d\lambda+\beta_2^2\,d\mu\otimes d\mu,
\qquad S=\theta_1\,d\lambda\otimes d\lambda+\theta_2\,d\mu\otimes
d\mu.
\end{align}
Applying Theorems~\ref{prop_nec_cond}, \ref{thm_loc_ex},
and~\ref{thm_uniq}, we arrive at the following conclusions.

\begin{proposition}\label{prop_torus}
Suppose $T$ is a smooth tensor on $\mathcal T_x$ given by
formula~\eqref{torus_t}. Assume $\zeta(1)$ is not equal to~$0$. Let
$R$ and $S$ be tensors on $\partial\mathcal T$ given
by~\eqref{torus_RS}. Then the following statements are equivalent:
\begin{enumerate}
\item
For some $\epsilon_0\in(0,x)$, there exists a rotationally symmetric
metric $G$ on $\mathcal T_{\epsilon_0}$ satisfying the equation
$\Ric(G)=T$ on $\mathcal T_{\epsilon_0}$ together with the boundary
conditions $G_{\partial\mathcal T}=R$ and $\II(G)=S$
on~$\partial\mathcal T$.
\item
The quantity
\begin{align*}
\frac1{\zeta(1)}\,\left(\frac1{\beta_1^2}\psi_1(1)+\frac1{\beta_2^2}\psi_2(1)+2\frac{\theta_1\theta_2}{\beta_1^2\beta_2^2}\right)
\end{align*}
is positive.
\end{enumerate}
If the above statements hold and $\zeta(\gamma)\ne0$ for
$\gamma\in(1-\epsilon_0,1]$, then the rotationally symmetric metric
$G$ satisfying $\Ric(G)=T$, $G_{\partial\mathcal T}=R$, and
$\II(G)=S$ is unique on $\mathcal T_{\epsilon_0}$.
\end{proposition}

We end this section with a remark concerning the solvability of the
prescribed Ricci curvature equation on all of $\mathcal T$.
Generally speaking, this is a complicated matter. We will not
discuss it thoroughly.

\begin{remark}\label{rem_global}
Suppose $\breve T$ is a smooth positive-definite (0,2)-tensor on
$\mathcal T$ such that the restriction of $\breve T$ to~$\mathcal
T_x$ is a rotationally symmetric tensor on $\mathcal T_x$. Assume
$\theta_1,\theta_2>0$. Let $S$ be given by the second formula
in~\eqref{torus_RS}. Then, for some $\breve\epsilon_0\in(0,x)$,
there exists a rotationally symmetric metric $\breve G$ on $\mathcal
T_{\breve\epsilon_0}$ such that $\Ric(\breve G)=\breve T$
on~$\mathcal T_{\breve\epsilon_0}$ and $\II(\breve G)=S$ on
$\partial\mathcal T$. This is an easy consequence of
Proposition~\ref{prop_torus}. However, there is no smooth Riemannian
metric on all of $\mathcal T$ such that its Ricci curvature equals
$\breve T$ and the second fundamental form of~$\partial\mathcal T$
in this metric with respect to the outward unit normal equals $S$.
This follows from Theorem~2 in~\cite{HW99}.
\end{remark}

\section{The Einstein equation on a solid torus}\label{sec_Einstein}

In this section, we briefly discuss the Einstein equation on the
solid torus $\mathcal T$. Let us first state a definition. We say
that a smooth Riemannian metric $G$ on $\mathcal T$ is a
\emph{rotationally symmetric metric on $\mathcal T$} if
\begin{align}\label{rot_all_T}
G=p^2(\gamma)\,d\gamma\otimes d\gamma+g_1^2(\gamma)\,d\lambda\otimes
d\lambda+g_2^2(\gamma)\,d\mu\otimes d\mu,\qquad \gamma\in(0,1],
\end{align}
in our cylindrical coordinates. Here, $p$ and $g_1,g_2$ have to be
functions from $(0,1]$ to $(0,\infty)$. The solutions of the
Einstein equation in the class of rotationally symmetric metrics on
$\mathcal T$ can be found explicitly through a simple computation.
We write down these solutions below. Before doing so, however, we
need to make a few preparatory comments.

Suppose $G$ is a rotationally symmetric metric on $\mathcal T$ and
equality~\eqref{rot_all_T} is satisfied. It is well known that the
function $g_1$ must then admit a smooth odd extension to $[-1,1]$.
We will preserve the notation $g_1$ for this extension. Clearly,
$g_1(0)$ is equal to~0. This fact will be essential to our arguments
later on. The functions $p$ and $g_2$ must admit smooth even
extensions to $[-1,1]$. Again, we preserve the notations $p$ and
$g_2$ for these extensions. It is clear that $p'(0)$ and $g_2'(0)$
are equal to~0 (the prime now designates differentiation with
respect to $\gamma$). This will also be important to our arguments
later. The derivative $g_1'(0)$ must equal $2\pi p(0)$. The values
$p(0)$ and $g_2(0)$ need to be positive. Thus, we have listed
several properties of the components of a rotationally symmetric
metric on $\mathcal T$ in our cylindrical coordinates. Conversely,
suppose we have three smooth functions $p$ and $g_1,g_2$ from
$[-1,1]$ to $\mathbb R$. Assume that $g_1$ is positive on $(0,1]$
and odd whereas $p$ and $g_2$ are positive on $[-1,1]$ and even. If
$g_1'(0)=2\pi p(0)$, then equality~\eqref{rot_all_T} defines a
rotationally symmetric metric on~$\mathcal T$. The reader may wish
to see, e.g.,~\cite{RH84,JDDKAY10} for related material.

In what follows, we fix a number $\tau\in\mathbb R$ and set
$\kappa=\sqrt{2|\tau|}\,$. Consider a rotationally symmetric metric
$G$ on $\mathcal T$ satisfying~\eqref{rot_all_T}. The notation
$\Ric(G)$ will stand for the Ricci curvature of $G$. Let us
introduce a new coordinate system $(\xi,\lambda,\mu)$ on $\mathcal
T$. We obtain it from $(\gamma,\lambda,\mu)$ by replacing the
parameter $\gamma$ with the parameter $\xi$ connected to $\gamma$ by
the formula
\begin{align}\label{def_s}
\xi(\gamma)=\int_0^\gamma p(\rho)\,d\rho,\qquad \gamma\in[0,1].
\end{align}
The values of $\xi$ range from 0 to the number $\xi_0$ equal to
$\int_0^1p(\rho)\,d\rho$. In the coordinate system
$(\xi,\lambda,\mu)$, we have
\begin{align}\label{rot_T_s}
G=d\xi\otimes d\xi+\bar g_1^2(\xi)\,d\lambda\otimes d\lambda+\bar
g_2^2(\xi)\,d\mu\otimes d\mu,\qquad\xi\in(0,\xi_0],
\end{align}
where $\bar g_1,\bar g_2$ are smooth functions from $[0,\xi_0]$ to
$[0,\infty)$. The comments made in the previous paragraph imply
\begin{align}\label{core_new_coord}
\bar g_1(0)=0,~\dot{\bar g}_1(0)=2\pi,~\dot{\bar g}_2(0)=0.
\end{align}
Here and in what follows, the dot designates differentiation with
respect to~$\xi$.

\begin{proposition}\label{prop_Einstein_+}
Suppose the rotationally symmetric metric $G$ on $\mathcal T$
satisfies equality~\eqref{rot_all_T} and solves the Einstein
equation
\begin{align}\label{Einstein_eq_+}
\Ric(G)=\tau G
\end{align}
on $\mathcal T$. If $\tau>0$, then $\xi_0$ lies in
$\big(0,\frac\pi\kappa\big)$ and there exists a constant $c_1>0$
such that
\begin{align}\label{met_E+}
G=d\xi\otimes d\xi+\frac{8\pi^2}{\kappa^2}(1-\cos\kappa
\xi)\,d\lambda\otimes d\lambda+c_1(1+\cos\kappa \xi)\,d\mu&\otimes
d\mu,\qquad \xi\in(0,\xi_0].
\end{align}
If $\tau<0$, then we can find $c_2>0$ ensuring that
\begin{align}\label{met_E-} G=d\xi\otimes d\xi+\frac{16\pi^2}{\kappa^2}\sinh^2\frac{\kappa\xi}2\,d\lambda\otimes d\lambda
+c_2\cosh^2\frac{\kappa\xi}2\,d\mu&\otimes d\mu,\qquad
\xi\in(0,\xi_0].
\end{align}
When $\tau=0$, there is $c_3>0$ such that
\begin{align}\label{met_E0}
G=d\xi\otimes d\xi+4\pi^2\xi^2\,d\lambda\otimes
d\lambda+c_3\,d\mu&\otimes d\mu,\qquad \xi\in(0,\xi_0].
\end{align}
\end{proposition}

\begin{proof}
We will only consider the case where $\tau>0$. Analogous arguments
work if $\tau<0$ or $\tau=0$. The metric $G$ can be written in the
form~\eqref{rot_T_s}. Our goal is to find the functions $\bar g_1$
and $\bar g_2$. The Einstein equation~\eqref{Einstein_eq_+} and
Lemma~\ref{lem_Ricci_curv} imply
\begin{align}\label{eq_coord_Einstein}
-\bar g_2\ddot{\bar g}_1-\dot{\bar g}_1\dot{\bar g}_2&=\tau\bar g_1\bar g_2,\notag\\
-\bar g_1\ddot{\bar g}_2-\dot{\bar g}_1\dot{\bar g}_2&=\tau\bar
g_1\bar g_2, \qquad\xi\in(0,\xi_0].
\end{align}
Adding these together yields
\begin{align*}
(\bar g_1\bar g_2)\,\ddot{}+2\tau\bar g_1\bar g_2=0,\qquad
\xi\in(0,\xi_0].
\end{align*}
In view of~\eqref{core_new_coord}, we conclude that
\begin{align*}
\bar g_1(\xi)\bar g_2(\xi)=\frac{2\pi\bar g_2(0)}\kappa\sin\kappa
\xi,\qquad\xi\in[0,\xi_0].
\end{align*}
It now follows from~\eqref{eq_coord_Einstein}
and~\eqref{core_new_coord} that
\begin{align*}
\dot{\bar g}_1(\xi)\bar g_2(\xi)&=\pi\bar g_2(0)(\cos\kappa\xi+1), \\
\bar g_1(\xi)\dot{\bar g}_2(\xi)&=\pi\bar g_2(0)(\cos\kappa
\xi-1),\qquad\xi\in[0,\xi_0].
\end{align*}
Manipulating the last three equalities and making use
of~\eqref{core_new_coord} again, we obtain
\begin{align*}
\bar
g_1(\xi)&=\frac{2\sqrt2\,\pi\sin\kappa\xi}{\kappa\sqrt{1+\cos\kappa\xi}}\,, \\
\bar g_2(\xi)&=\frac{\bar g_2(0)}{\sqrt2}\,\sqrt{1+\cos\kappa
\xi}\,,\qquad \xi\in[0,\xi_0].
\end{align*}
It becomes clear that $\xi_0$ must lie in
$\big(0,\frac\pi\kappa\big)$ and the metric $G$ must
satisfy~\eqref{met_E+}.
\end{proof}

\begin{remark}
Proposition~\ref{prop_Einstein_+} enables us to make conclusions
about the existence of solutions to the Einstein equation on
$\mathcal T$ with given boundary data. For example, suppose $G$ is a
rotationally symmetric metric on $\mathcal T$. Fix two numbers
$\beta_1,\beta_2>0$. Let $G_{\partial\mathcal T}$ denote the
Riemannian metric on $\partial\mathcal T$ induced by~$G$.
Proposition~\ref{prop_Einstein_+} implies that, if $\tau>0$ and
$\beta_1\ge\frac{4\pi}{\kappa}$, equation~\eqref{Einstein_eq_+} and
the equality
\begin{align*}
G_{\partial\mathcal T}=\beta_1^2\,d\lambda\otimes
d\lambda+\beta_2^2\,d\mu\otimes d\mu
\end{align*}
cannot be satisfied simultaneously. This observation is somewhat
related to the material in~\cite{MA08,MA09}.
\end{remark}

The converse of Proposition~\ref{prop_Einstein_+} holds as well.
More precisely, suppose $p$ is a smooth positive even function on
$[-1,1]$. Introduce the coordinate system $(\xi,\lambda,\mu)$ on
$\mathcal T$ by replacing the parameter $\gamma$ in the coordinate
system $(\gamma,\lambda,\mu)$ with the parameter $\xi$ related to
$\gamma$ through~\eqref{def_s}. If $\tau>0$ and
$\xi_0\in\big(0,\frac\pi\kappa\big)$, then formula~\eqref{met_E+}
defines a rotationally symmetric metric on~$\mathcal T$ for any
$c_1>0$. It is not difficult to verify that this metric solves the
Einstein equation~\eqref{Einstein_eq_+}. In doing so,
Lemma~\ref{lem_Ricci_curv} may come in handy. Let us now assume
$\tau<0$. Formula~\eqref{met_E-} determines a rotationally symmetric
metric on $\mathcal T$ for any $c_2>0$. One easily checks that this
metric satisfies~\eqref{Einstein_eq_+}. Finally, let us assume
$\tau=0$. We arrive at similar conclusions. Namely,~\eqref{met_E0}
yields a rotationally symmetric metric on $\mathcal T$ for any
$c_3>0$. This metric solves~\eqref{Einstein_eq_+}.

\section*{Acknowledgements}

I express my gratitude to Andrea Young for the many stimulating
discussions on the topics covered in the present paper. I am also
thankful to Dan Knopf for our conversations about the cross
curvature flow on a solid torus. To a large extent, those
conversations inspired the work above. Finally, I thank Ben Babao
for pointing out a computational error in an earlier version of the
paper and also Ye Kai Wang and the referee for suggesting
Remark~\ref{rem_Gauss-Codazzi} to me.

\end{document}